\documentclass[11pt,a4paper]{amsart}
\usepackage[left=3.5cm,right=3.5cm,top=3.5cm,bottom=3.5cm]{geometry}
\usepackage{enumerate}
\usepackage{leftidx}
\usepackage{mathtools}
\usepackage{amsmath,amscd,amssymb,amsthm}
\usepackage[arrow,matrix]{xy}
\usepackage[OT2,T1]{fontenc}
\usepackage{booktabs}
\usepackage[position=bottom]{subfig}
\usepackage{lipsum}
\usepackage{graphicx}
\usepackage{setspace,color}
\usepackage{hyperref}
\usepackage{faktor, mathrsfs, tikz-cd}

\tikzset{
  symbol/.style={
    draw=none,
    every to/.append style={
      edge node={node [sloped, allow upside down, auto=false]{$#1$}}}
  }
}

\DeclareMathOperator{\Gal}{Gal}
\DeclareMathOperator{\Hom}{Hom}
\DeclareMathOperator{\PG}{PG}
\DeclareMathOperator{\AG}{AG}
\DeclareMathOperator{\Aut}{Aut}

\theoremstyle{definition}
\newtheorem{definition}{Definition}[section]
\newtheorem{example}[definition]{Example}
\newtheorem{remark}[definition]{Remark}
\newtheorem*{remark*}{Remark}

\theoremstyle{plain}
\newtheorem{theorem}[definition]{Theorem}
\newtheorem{corollary}[definition]{Corollary}
\newtheorem{lemma}[definition]{Lemma}
\newtheorem{proposition}[definition]{Proposition}

\newcommand{\vF}{{ \mathbb F }}
\newcommand{\vK}{{ \mathbb K }}
\newcommand{\vP}{{ \mathbb P }}
\newcommand{\vA}{{ \mathbb A }}

\author[L. Bastioni]{Luca Bastioni}
\address{University of South Florida \\ 4202 E Fowler Ave\\
33620 Tampa, US.}
\email{lbastioni@usf.edu}

\author[G. Micheli]{Giacomo Micheli}
\address{University of South Florida\\
4202 E Fowler Ave\\
33620 Tampa, US.
}
\email{gmicheli@usf.edu}
\title{On complete $m$-arcs}

\makeatletter
\@namedef{subjclassname@2020}{%
  \textup{2020} Mathematics Subject Classification}
\makeatother
\begin{document}

\begin{abstract}
Let $m$ be a positive integer and $q$ be a prime power.
For large finite base fields $\vF_q$, we show that any curve can be used to produce a complete $m$-arc as long as some generic explicit geometric conditions on the curve are verified. To show the effectiveness of our theory, we derive complete $m$-arcs from hyperelliptic curves and from Artin-Schreier curves.
\end{abstract}

\maketitle

\section{Introduction}

Let $q$ be a prime power and $\vF_q$ be the finite field of order $q$. Let $\vP^2(\vF_q)=\PG(2,q)$ be the set of $\vF_q$-points $[X_0,X_1,X_2]$ of the projective plane. 
Let $\AG(2,q)$ be the set of $\vF_q$ points of the affine plane, which we will identify with the projective points of $\vP^2(\vF_q)$ with $X_2\neq 0$. 
A $(k,m)$\emph{-arc} $\mathcal A$ in $\vP^2(\vF_q)$ is a set of $k$ points such that there are no $m+1$ points that are collinear and such that there exist $m$ collinear points. For fixed $m,q$, we have that the set $\mathcal A_{m}(q)$ of all $m$-arcs in $\vP^2(\vF_q)$ is partially ordered by inclusion. A \emph{complete} $m$-arc is a maximal element in $\mathcal A_m(q)$, i.e. it is an $m$-arc that is not contained in a strictly larger $m$-arc.

In the late 50's, Segre introduced the notion of arc for $m=2$, which has been has been extensively studied in the literature (Segre was in fact interested to know how many points can complete arcs in $\vP^2(\vF_q)$ have, see \cite{Segre1959LeGD}). In this simpler case $m=2$, the theory is well developed and there are plenty fo constructions (see for example \cite{Anbar2013SMALLCC, Anbar2015SMALLCC2, Anbar2013Bicovering, Hirschfeld1986Arcs, hirschfeld1998projective, hirschfeld1998packing, Szhonyi1989Survey, szonyi1997BookApplication}). For maximum size arcs in odd characteristic, Segre himself proved that the set of points of an arc is a conic.

Concerning minimal sizes, much less is known. Using a probabilistic method, the paper \cite{Kim2003SmallCA}, shows that there exists a complete arc of size $O(\sqrt{q}\log^c q)$ in a projective plane of order $q>M$ (for some positive constants $c,M$ independent on $q$).
Another clever idea to construct arcs (due to Segre and Lombardo-Radice \cite{Lombardo1956Dispari, Segre1962Sigma}) is to use a small set of $\vF_q$-rational points of a curve over  a finite field. Many research has been done in this direction \cite{Korchmaros1983Examples, Voloch1990Completeness, Voloch1991NonSquare, Schoof1987nonsingular, Szonyi1985Small, Giulietti2002Cubic} and the currently best known bound using this method gives $O(q^{3/4})$

The case $m>2$ is much more complex and the currently known constructions either use the tehory of $2$-character sets (see \cite[Sects. 12.2 and 12.3]{hirschfeld1998projective} and \cite{Hamilton2001Sets})  use or use the theory of certain special family of curves \cite{Giulietti2007Completeness, Hirschfeld1988Characterization, Tallini1970Graphic, Bartoli2017Completeness}.
For $m = 3$ and $m=4$ \cite{Bartoli2016K3, Bartoli2017K4} construct small complete $m$-arcs of size $o(q)$ using subsets of the $\mathbb{F}_q$-rational points of a curve of degree $m+1$.
The paper \cite{Bartoli2022Complete} construct a family of curves that give rise to $m$-arcs of size smaller than $q$ over certain subsequence of extension fields for all $m>8$. In fact, the authors show that they can construct a curve and a sequence of positive integers $n_k$ such that the number of $\vF_{q^{n_k}}$ rational points of $\mathcal C$ give rise to an $m$-arc of size roughly $q^{n_k}-C\sqrt{q^{n_k}}$ for an explicit $C$. Using a similar Galois theoretical approach as \cite{Bartoli2022Complete}, in \cite{korchmaros2023algebraic} the authors were able to construct complete $m$-arcs using the Hermitian curve of degree $q+1$ and the BKS curve of degree $q+1$.

The purpose of this paper is to provide a general theory to construct families of complete $m$-arcs arising from curves that satisfy explicit and easy to verify generic conditions.
As a proof of concept, we employ our methods in the case of hyperelliptic curves and Artin-Schreier curves, showing that they can be used to construct a complete $m$-arcs.

Apart from their fundamental interest in combinatorics, complete $m$-arcs of size $k$ have an important connection with coding theory, as they correspond to codes over $\vF_q$ of length $k$, dimension $3$ and distance $k-m$ (in the coding theoretical notation, $[k,3,k-m]_q$-codes) that cannot be extended to a code with same dimension and larger minimum distance. In fact, given an $m$-arc $A$, one can place as columns of a $3\times k$ matrix $G$ representatives in $\vF_q^3$ of the points of  the $m$-arc. Let $\mathcal C$ be the code generated by $G$. Now, for every $x\in \vF_q^3$, the codeword $xG$ has weight at least $k-m$ or otherwise one could find $m+1$ collinear points in $A$.

A general introduction to $(k, m)$-arcs can be found in the monograph \cite[Chapter 12]{sherk1981jwp}, as well as in the survey paper \cite[Section 5]{hirschfeld1998packing}.

\subsection{Outline of the paper}
We fist provide an introduction to the basic techniques and notations that we use in the paper 
(Subsection \ref{subsec:notation} and Section \ref{subsec:back}).
Then,  we describe the auxiliary results that we will use to convert geometric properties into the arithmetic ones (Section \ref{sec:AGUP}). Then, we provide the main result (Section \ref{sec:mainresult}). Finally, to show the effectiveness of our results, we give two applications (Section \ref{sec:applications}) to hyperelliptic and Artin-Schreier curves.

\subsection{Notation}\label{subsec:notation}
Let $x$ be a transcendental element over $\vF_q$ and $\vF_q(x)$ be the rational function field in the variable $x$.
For a curve $\mathcal C$ we denote by  $\mathcal C (\vF_q)$  the $\vF_q$-rational points  of $\mathcal C$, and by $\vF_q(\mathcal C)$ its function field over $\vF_q$. For a polynomial $F(X_0,\ldots,X_n)$ we denote by $\deg(F)$ the total degree of $F$, whereas $\deg_{X_i}(F)$ denotes the maximal degree of $X_i$ occurring in $F$. In the following we denote by $\vP^2(\vF_q)$ the $\vF_q$ rational points of the projective plane $\vP^2$ over $\vF_q$, and with $\vA^2(\vF_q)=\vP^2(\vF_q)\setminus \ell_{\infty}$ the $\vF_q$ rational points of the affine plane $\vA^2$, where $\ell_{\infty}$ is the line at infinity $[X,Y,0]$.

Let $F=F(X_0,X_1,X_2)=0$ be the homogeneous equation of an irreducible curve $\mathcal C$ of degree $n$ in $\mathbb{P}^2$. For $i\in \{1,2,3\}$ let $F_i=\frac{\partial}{\partial X_i}F(X_0,X_1,X_2)$ be the partial derivative of $F$ with respect the $i$-th variable. Then the dual curve $\mathcal C^*$ of $\mathcal C$ is defined as the closure of the image of the Gauss map
\begin{align*}
    \mathcal{G}:\mathcal C & \rightarrow \vP^2 \\
    (X_0,X_1,X_2) & \mapsto (F_0(X_0,X_1,X_2),F_1(X_0,X_1,X_2),F_2(X_0,X_1,X_2))
\end{align*}
For a planar irreducible curve $\mathcal C$, let $g(x,y)=0$ be its affine equation.
We say that $\vF_q(\mathcal C)$ is isomorphic to $\vF_q(\mathcal C^*)$ via $\mathcal G$, and write $\vF_q(\mathcal C)\cong_{\mathcal G} \vF_q(\mathcal C^*)$, if
\begin{equation}\label{eq:gaussmapfunc}
\vF_q(x,y)=\vF_q\left(\frac{\partial_x g(x,y)}{\partial_y g(x,y)}, x\frac{\partial_x g(x,y)}{\partial_y g(x,y)}+y \right).
\end{equation}
Since $\mathcal C$ is a planar irreducible curve, we have that $\mathcal C^*$ is also  planar and irreducible.
Observe that $\mathcal G(\mathcal C)$ is closed when $\mathcal C$ is non-singular. 

\section{Background}\label{subsec:back}
\subsection{Arcs}
A $(k,m)$-arc $\mathcal A$ in $\vP^2(\vF_q)$ is a set of $k$ points no $m+1$ of which are collinear and such that there exist $m$ collinear points. The arc $\mathcal A$ is called a \textit{complete} $(k,m)$-arc if it is maximal with respect to the set theoretical inclusion among all $m$-arcs. In particular, every point in the complement of a complete $m$-arc $\mathcal A$ is aligned with $m$ points in $\mathcal A$.

\subsection{Rational maps between curves}\label{morphism}
We recall in this part some facts about rational maps and birational morphisms. For a complete exposition of the theory see for example \cite{shafarevich, silverman2009arithmetic, stichtenoth}. 

Let $\vK$ be an algebraically closed field and $\mathcal V\subset \vP^n(\vK)$ a projective variety, i.e. an irreducible algebraic set of zeros in $\vK$ of homogeneous polynomials with coefficients in $\vK$. In the notation of this paper, an algebraic curve is simply a projective variety of dimension one.
Let $\vK(\mathcal V)$ (resp. $\vK[\mathcal V]$) denote the function field (resp. the homogeneous coordinate ring) of $\mathcal V$. A rational function $\varphi\in\vK(\mathcal V)$ is called \textit{regular} at $x\in \mathcal V$ if it can be written as $\varphi=\frac{f}{g}$ where $f,g\in\vK[\mathcal V]$ and $g(x)\neq 0$. The value of $\varphi$ at the point $x$ is $\frac{f(x)}{g(x)}$ and it is denoted by $\varphi(x)$. A rational function that is regular at every $x\in \mathcal V$ is called \textit{regular function}.

Let $\mathcal W\subset\vP^m(\vK)$ a projective variety. A \textit{rational map} $\varphi:\mathcal V\dashrightarrow \mathcal W$ is a $m$-tuple of $\varphi_0,\ldots,\varphi_m\in\vK(\mathcal V)$ such that $\varphi(x)=[\varphi_0(x),\ldots,\varphi_m(x)]\in \mathcal W$ for all points $x\in \mathcal V$ at which all the $\varphi_i$ are regular. The rational map $\varphi$ is said to be \textit{regular} at the points $x\in \mathcal V$ where it is defined. A rational map that is regular at every point $x\in \mathcal V$ is called a \textit{morphism}.

The following result states that, for a non-singular curve $\mathcal C$ a rational map is defined at every point (see \cite[Proposition 2.1]{silverman2009arithmetic}).

\begin{proposition}\label{propregular}
Let $\mathcal C$ be an algebraic curve and $\mathcal W\subset \vP^m$ be a variety. Let $P\in\mathcal C$ be a non-singular point and let $\varphi:\mathcal C\dashrightarrow \mathcal W$ be a rational map. Then $\varphi$ is regular at $P$. In particular, if $\mathcal C$ is non-singular, then $\varphi$ is a morphism.
\end{proposition}

A morphism $\varphi:\mathcal V\rightarrow \mathcal W$ is said to be an \textit{isomorphism} if there exists a morphism $\psi:\mathcal W\rightarrow \mathcal V$ such that $\varphi\circ\psi$ and $\psi\circ\varphi$ are the identity maps on $\mathcal W$ and $\mathcal V$ respectively. Two varieties $\mathcal V$ and $\mathcal W$ are said to be \textit{birationally equivalent} (or just \textit{birational}) if there exist rational maps $\varphi:\mathcal V\dashrightarrow \mathcal W$ and $\psi:\mathcal W\dashrightarrow \mathcal V$ such that $\varphi\circ\psi$ and $\psi\circ\varphi$ are the identity maps on the points of $\mathcal W$ and $\mathcal V$ where they are defined. Each rational map $\varphi$ and $\psi$ is called a \textit{birationally equivalence}. If two varieties are isomorphic they are birationally equivalent, but the converse is generally not true.

\begin{remark}\label{DefOnNSingular}
If $\varphi:\mathcal V\dashrightarrow \mathcal W$ is a birationally equivalence, then $\varphi$ is a 1-to-1 correspondence everywhere except at non-regular points of $\mathcal V$ and $\mathcal W$. Therefore, as a consequence of Proposition \ref{propregular}, if $\mathcal C$ is a non-singular curve, then $\varphi:\mathcal C\dashrightarrow \mathcal W$ is defined everywhere on $\mathcal C$, and 1-to-1 only at \textit{regular} points of $\mathcal W$. This means that in the particular case when $\mathcal W$ is a curve too, then $\phi$ is defined everywhere on $\mathcal C$, and 1-to-1 only at \textit{non-singular} points of $\mathcal W$.
\end{remark}

\subsection{Non-singular model of a curve}\label{model}
For the results in this subsection we refer to \cite[Appendix B]{stichtenoth}. Let $\mathcal V$ be a projective curve. Then there exists a non-singular projective curve $\mathcal V_{des}$ and a birational morphism $\pi:\mathcal V_{des}\rightarrow \mathcal V$ that satisfies the following universal property: if $\mathcal V'$ is another non-singular projective curve, and $\pi':\mathcal V'\rightarrow \mathcal V$ is a birational morphism, then there exists a unique isomorphism $\phi:\mathcal V_{des}\rightarrow \mathcal V'$ such that $\pi=\pi'\circ\phi$. This is to say that the following diagram commutes
\begin{center}
\begin{tikzcd}
\mathcal V_{des} \arrow[r, "\pi"] \arrow[rd, swap, "\phi"] &  \mathcal V\\
 
 & \mathcal V' \arrow[u,  swap, "\pi' "]
\end{tikzcd}
\end{center}
The pair $(\mathcal V_{des},\pi)$, or simpler just $\mathcal V_{des}$, is called the non-singualr model of $\mathcal V$.

\subsection{Galois Theory over Global Fields}
In this section we use the notation and terminology of \cite{stichtenoth}. 
For a field $M$ we denote by $\Aut(M)$ the automorphism group of $M$.
Let $L\supseteq K$ be a separable field extension and let $M$ be the Galois closure.
We denote by $\Gal(M:K)$ the Galois group of $M:K$, i.e. $\{g\in \Aut(M): \;g(x)=x \;\forall x \in K\}$.
The next result is a consequence of \cite[Satz 1]{derWaerden}.
\begin{lemma}[\textbf{Orbits' Lemma}]\label{orbits}
Let $L:K$ be a finite separable extension of function fields, let $M$ be its Galois closure and $G:=\Gal(M:K)$ be its Galois group. Let $P$ be a place of $K$ and $\mathcal Q$ be the set of places of $L$ lying above $P$. Let $R$ be a place of $M$ lying above $P$. Then we have the following:
\begin{itemize}
	\item[1)] There is a natural bijection between $\mathcal Q$ and the set of orbits of $H:=\Hom_K(L,M)$ under the action of the decomposition group $D(R|P)=\lbrace g\in G \;|\; g(R)=R\rbrace$.
	\item[2)] Let $Q\in\mathcal Q$ and let $H_Q$ be the orbit of $D(R|P)$ corresponding to $Q$. Then $|H_Q|=e(Q|P)f(Q|P)$ where $e(Q|P)$ and $f(Q|P)$ are the ramification index and relative degree, respectively.
	\item[3)] The orbit $H_Q$ partitions further under the action of the inertia group $I(R|P)=\lbrace g\in D(R|P) \;|\; v_R(g(s)-s)\geq 1 \; \forall s\in\mathcal O_R \rbrace$ into $f(Q|P)$ orbits of size $e(Q|P)$.
\end{itemize}
\end{lemma}

The next result gives a characterization of the Galois group using the inertia groups. For a reference see for example \cite{stichtenoth}.
\begin{lemma}\label{InertiaLemma}
Let $L:K$ be a finite separable extension of function fields, let $M$ be its Galois closure and let $G:=\Gal(\bar\vF_q M:\bar\vF_q K)$. Then $G$ is generated by the inertia groups $I(R|P)$, i.e.
\begin{equation*}
G=\langle I(R|P) : P\text{ place of }K,\; R|P,\; R\text{ place of }M \rangle
\end{equation*}
\end{lemma}

The following result follows from \cite[Proposition 1]{hashimoto}.
\begin{proposition}\label{PropTransp}
Let $G$ be a transitive subgroup of $S_n$ generated by transpositions. Then $G=S_n$.
\end{proposition}

\section{From geometry to Galois theory}\label{sec:AGUP}
For this section, let $\vK$ be an algebraically closed field, let $\vP^2=\vP^2(\vK)$, and let $r$ be a line in $\vP^2$. Remark \ref{rem:extgalgeo} combined with Chebotarev Density Theorem explains one way to transfer geometric information obtained over 
$\vP^2(\overline \vF_q)$ to arithmetic information in $\vP^2(\vF_q)$, which will be needed to construct $m$-arcs in $\vP^2(\vF_q)$ in Section \ref{sec:mainresult}.
\begin{definition}
The line $r\subseteq \vP^2$ is said to be a \emph{bitangent} to a curve $\mathcal C$ if it has at least two points of intersection with $\mathcal C$ where it is tangent to $\mathcal C$. 
\end{definition}

We start with a preliminary lemma.
\begin{lemma}\label{lemma0}
Let $\mathcal C\subseteq \vP^2$ be a planar, irreducible, projective curve of degree $m$ over a finite field $\vF_q$ with singular points contained in the line at infinity. Let $\mathcal C^*$ be the dual curve of $\mathcal C$ and suppose $\vF_q(\mathcal C)\cong_{\mathcal G}\vF_q(\mathcal C^*)$, that is $\vF_q(\mathcal C)$ and $\vF_q(\mathcal C^*)$ are isomorphic through the Gauss map $\mathcal G$. Then there are at most $\mathcal O(m^4)$ affine bitangents to the curve.
\end{lemma}
\begin{proof}
Let $F=F(X_0,X_1,X_2)=0$ be the homogeneous equation of $\mathcal C$ and $G=G(X_0,X_1,X_2)=0$ be the homogeneous equation of $\mathcal{C}^*$. Since $F$ has degree $m$, then $G$ has at most degree $m(m-1)$ (see for example \cite{nguyen2016plucker}).

Since $\vF_q(\mathcal C)\cong_{\mathcal G}\vF_q(\mathcal C^*)$, the Gauss map $\mathcal{G}:\mathcal{C}\cap\vA^2\rightarrow\mathcal C^*$ can be extended to a birational map $\mathcal G'$ between $\mathcal C$ and $\mathcal C^*$.

Let  $\mathcal C_{des}$ and $\mathcal C^*_{des}$  be the non-singular models of $\mathcal C$ and $\mathcal C^*$ respectively. Let $\pi_1:\mathcal C_{des}\rightarrow\mathcal C$ and $\pi_2:\mathcal C^*_{des}\rightarrow\mathcal C^*$ the birational morphisms attached to the non-singular models (satisfying the universal property shown in Subsection \ref{model}, and defined everywhere as explained in Subsection \ref{morphism}, notice in particular Remark \ref{DefOnNSingular}).

Notice that, thanks to Proposition \ref{propregular} the birational  map  $\mathcal G' \circ \pi_1$ is a birational morphism because $\mathcal C_{des}$ is non-singular.

By using the universal property of $\pi_2$ for the birational morphism $\mathcal G' \circ \pi_1:\mathcal C_{des} \dashrightarrow \mathcal C^*$ between $\mathcal C_{des}$ and $\mathcal C^*_{des}$ one obtains an isomorphism $\mathcal G_{des}$ that makes the diagram commutative.
\begin{center}
\begin{tikzcd}
\mathcal{C}\cap\mathbb{A}^2 \arrow[rrd, "\mathcal G"] \arrow[d, symbol=\subseteq]  & &  \\

\mathcal{C} \arrow[rr,dashed, "\mathcal G'"] & & \mathcal{C}^*\\

 & & \\

\mathcal{C}_{des} \arrow[uu,  "\pi_1"] \arrow[rr, "\mathcal G_{des}"]  & & \mathcal{C}^*_{des} \arrow[uu,  "\pi_2"]
\end{tikzcd}
\end{center}

The number of affine bitangents to $\mathcal C\cap \vA^2$ equals the number of points of $\mathcal C^*$ that have multiple preimages through $\mathcal G$. 
Our strategy is to  prove that the points of $\mathcal C^*$ that have multiple preimages via $\mathcal G$ are contained in the points of $\mathcal C^*$ that are singular. We will then bound the singular points with Bezout's Theorem. 
 Suppose that $P$ is a point of $\mathcal C^*$ that has a set $T\subseteq \mathcal C\cap \vA^2$ of preimages via $\mathcal G$ of size larger than or equal to $2$. 
 Let now $T'=\pi_1^{-1}(T)\subseteq \mathcal C_{des}$. Since the diagram commutes, we have that $\pi_2\circ \mathcal G_{des}(T')=\{P\}$, and therefore $\pi_2$ is mapping a set of points of size larger than $1$ to a single point $P$, so $P$ must be a singular point of $\mathcal C^*$.

 We can now estimate the number of singular points of $\mathcal{C}^*$ using Bezout's Theorem: these points are contained in the set of solutions of the equations $G=0$ and $G_0=\frac{\partial}{\partial X_0}G=0$, which are at most $\mathcal O(\deg (G) \deg (G_0) )=\mathcal O(m^4)$ because $G_0$ cannot divide $G$ due to irreducibility.
\end{proof}

\begin{lemma}\label{lemma1}
Let $\mathcal C\subseteq \vP^2$ be a planar, irreducible, non-singular, projective curve of degree $m$ over a finite field $\vF_q$ such that $\mathcal C$ is birationally equivalent to its dual curve $\mathcal C^*$ through the Gauss map. Then there are at most $\mathcal O(m^4)$ bitangents to the curve.
\end{lemma}
\begin{proof}
It is a direct consequence of Lemma \ref{lemma0} because $\mathcal{C}$ is non-singular. 
\end{proof}

\begin{remark}\label{rem:genericity}
Notice that for irreducible, non-singular, planar, projective  curves, being birationally equivalent to the dual curve is a generic condition as long as the characteristic is different from $2$ if the base field is large enough (see for example \cite[Theorem 5.90]{hirschfeld2008algebraic} and specifically the rank condition on the matrix in \cite[Equation 5.49]{hirschfeld2008algebraic}).
\end{remark}

\begin{definition}
Let $\mathcal C\subseteq \vP^2$ a planar curve. A line of $\vP^2$ that is tangent to $\mathcal C$ with multiplicity at least $3$ at a point $P\in \mathcal C$ is said to be an \textit{inflection tangent} and the point of tangency $P$ is said to be \textit{inflection point}.
\end{definition}

\begin{theorem}\label{HessianTheorem}
Let $\mathcal C$ be an irreducible projective planar curve of degree $m\geq 3$ over a field $K$ and $F=0$ its homogeneous equation. Let $H_F$ be the Hessian determinant associated to $F$. Let $p$ be the characteristic of $K$, and assume that either $p=0$ or $p>m$. Then  $H_F\not\equiv 0\mod(F)$, and the intersection of $F$ with its Hessian curve consists of the singular points of $\mathcal C$ and the inflection points of $\mathcal C$.
\end{theorem}
\begin{proof}
The claim follows using \cite[Theorem 9.7 (b)]{kunz2007} since $\mathcal C$ is irreducible.
\end{proof}

\begin{lemma}\label{lemma2}
Let $p$ be a prime number and $q$ a power of $p$. Let $\mathcal C\subseteq \vP^2$ be a planar, irreducible, projective curve of degree $m\geq 3$ over a finite field $\vF_q$. Then, if $p>m$ the number of inflection tangents to $\mathcal C$ are at most $\mathcal O(m^2)$.
\end{lemma}

\begin{proof}
Let $F=F(X_0,X_1,X_2)=0$ be a homogeneous equation of $\mathcal C$. Let $H_F$ be the Hessian determinant associated to $F$, therefore $\deg(H_F)\leq 3(m-2)$. Now, since $\mathcal C$ is irreducible we can use Theorem \ref{HessianTheorem} obtaining that the intersection between $\mathcal C$ and the set of points that vanish at $H_F$ consists only of singular points and inflection points of $\mathcal C$. Since by definition each inflection point has at most one inflection tangent, one can bound the number of inflection tangents with the number of inflection points and singular points of $\mathcal C$, that is $\mathcal O(\deg(F)\deg(H_F))=\mathcal O(m^2)$ as a consequence of Bezout's Theorem.

\end{proof}

The following is a particular instance of a classical fact, we include its proof for completeness.

\begin{lemma}\label{lemma3}
Let $p$ be a prime number, $q$ a power of $p$, and $K/\overline \vF_q$ be a function field over $\overline \vF_q$. Let $L:K$ be a proper finite separable extension of degree $n$ of global function fields, and $M$ be the Galois closure of $L:K$. Let $P$ be a place of $K$, $\mathcal Q_P$ be the set of places of $L$ lying above $P$, $e(Q|P)$ the ramification index of $Q\in\mathcal Q_P$ over $P$ and $f(Q|P)$ the relative degree. Suppose that for every $P$ and for every $Q\in\mathcal Q_P$ we have $e(Q|P)\leq 2$ and $e(\bar{Q}|P)=2$ at most once for some $\bar{Q}|P$. If $G:=\Gal(M:K)$ is transitive, then $G$ is the symmetric group $S_n$.
\end{lemma}
\begin{proof}
Using Proposition \ref{PropTransp} it is enough to prove that $G$ is generated by transpositions. By Lemma \ref{InertiaLemma}, since $G$ is generated by the inertia groups $I(R|P)$ where $R$ is a place of $M$ lying above $P$, it is enough to prove that the inertia groups contain only transpositions. Since we are working over the algebraic closure, the decomposition group $D(R|P)$ equals the inertia group $I(R|P)$. In particular $f(R|P)=f(R|Q)f(Q|P)=1$ and so $f(Q|P)=1$. By the Orbits' Lemma \ref{orbits} the action of $D(R|P)$ (i.e. $I(R|P)$) on $\Hom_K(L,M)$ gives orbits of order $e(Q|P)$, so each decomposition group acts either like the identity or a transposition. Since $G$ acts faithfully on $\Hom_K(L,M)$ and $D(R|P)\leq G$, $D(R|P)$ cannot contain more than one element that acts like a transposition on the set $\Hom_K(L,M)$, so either $D(R|P)=\lbrace 1_G \rbrace$ or $D(R|P)=\lbrace 1_G, g \;|\; g \text{ transposition } \rbrace$. This shows that $G$ is generated by transpositions, and therefore since $G$ is transitive we have $G=S_n$.
\end{proof}

\begin{remark}\label{rem:extgalgeo}
Notice that the lemma above can also be used to show that an \emph{arithmetic} Galois group is symmetric. In fact, if $M:K$ is a Galois extension of a function field $K/\vF_q$, then we have that $ \Gal( k_M M:  k_MK)= \Gal(\overline \vF_q M: \overline \vF_qK)$, where $k_M$ is the field of constants of $M$ (see for example \cite[Lemma 2.6]{micheli2019selection}). This implies that $ \Gal(\overline \vF_q M: \overline \vF_qK)$ is canonically contained in $\Gal(M:K)\leq S_{[M:K]}$. Therefore, if one can prove $\Gal(\overline \vF_q M: \overline \vF_qK)=S_{[M:K]}$, one also has $\Gal(M:K)= S_{[M:K]}$ and most importantly $k_M=\vF_q$.
\end{remark}

\section{Main Result}\label{sec:mainresult}

Let us now state and prove the main theorem.

\begin{theorem}\label{maintheorem}
Let $p$ be a prime number and $q$ a power of $p$. Let $\mathcal C\subseteq \vP^2$ be a planar, irreducible, projective curve of degree $m$ over a finite field $\vF_q$ such that $p>m$, the singular points of $\mathcal C$ are contained in the line at infinity, and $\mathcal C$ is birationally equivalent to its dual $\mathcal C^*$ through the Gauss map. There exists an explicit constant $c$ independent of $q$, for which if $q>c$ then there exists a set of points $S\subseteq \vP^2(\vF_q)$ of size at most $O(m^5)$ such that $\mathcal C(\vF_q)\cup S$ is a complete $m$-arc of $\vP^2(\vF_q)$.
\end{theorem}
\begin{remark}
Notice that the condition that $\mathcal C$ is birationally equivalent to its dual $\mathcal C^*$ is generic in characteristic different from $2$ (see Remark \ref{rem:genericity}). It is possible to calculate a sharp explicit constant $c$ by estimating via Hurwitz formula the genus of the Galois closure of a certain covering of curves that appears in the proof (see Remark \ref{rem:boundgenus}).
\end{remark}
\begin{proof}[Proof of Theorem \ref{maintheorem}]
Consider $\mathcal R=\lbrace r_1,\ldots,r_N \rbrace$ with $N={\mathcal O(m^4)}$ the set of all the bitangents and inflection tangents as in Lemma \ref{lemma0} and Lemma \ref{lemma2}.

Firstly, we prove that through any affine point outside these lines and outside $\mathcal C(\vF_q)$, there exists a line intersecting $\mathcal C(\vF_q)$ in exactly $m$ affine points. Afterwards, we outline how to choose a set $S$ of points on the lines in $\mathcal R$ to construct a complete $m$-arc $\mathcal C(\vF_q)\cup S$.

Therefore, consider a point $P=(a,b)\in\vA^2(\vF_q)$ such that $P\notin r_i$ for any $i\in \{1,\ldots,N\}$.
Let $f(x,y)=0$ an affine equation for $\mathcal C$, i.e. $f(x,y)\in\vF_q[x,y]$ irreducible of degree $m$. Let $\ell_{P,t}:y=t(x-a)+b$ be the (formal) line through $P$ with slope $t$ and 
\begin{equation*}
F_P(t,x)=f(x,t(x-a)+b)\in \vF_q[t,x]
\end{equation*}
be the polynomial obtained by $f(x,y)$, which results from substituting $y$ with $t(x-a)+b$. We want to prove that there exists a specialization $t_0$ for $t$ such that $F_P(t_0,x)$ totally splits, giving $m$ points of intersection between the curve $f(x,y)=0$ and the line $\ell_{P,t_0}$. To this end, consider $G=\Gal(F_P(t,x)\mid \overline{\vF}_q(t))$. Now notice that since $f(x,y)$ is irreducible, then $F_P(t,x)$ is an irreducible polynomial.
Set $K=\vF_q(t)$ and let $L$ be the fraction field of $\vF_q[x,t]/\langle F_P(t,x)\rangle$. Let $M$ be the Galois closure of $L:K$ and observe that   $\overline{\mathbb F}_q M$ is the Galois closure of
$\overline \vF_q L: \overline \vF_q K$. Now let $P$ be a place of $\overline \vF_q K$ and $\mathcal Q_P$ the set of places of $\overline \vF_q L$ lying above $P$. Because we are neither on a bitangent nor on an inflection tangent, $e(Q|P)\leq 2$ and $e(\bar{Q}|P)=2$ at most once for some $\bar{Q}|P$. Therefore, by Lemma \ref{lemma3} and Remark \ref{rem:extgalgeo}.
\begin{equation*}
\Gal(\overline{\mathbb F}_q M:\overline{\vF}_q(t))=\Gal(M:K)=S_m.
\end{equation*}
By \cite[Theorem 3.3]{Bartoli2022Complete}, there exists an explicit constant $c$ depending only on the genus of $M$ and the degree of $[L:K]$ for which if $q>c$ then $L:K$ has a totally split place, that means a $t_0$ exists such that $\ell_{P,t_0}$ intersect $\mathcal C$ in $m$ distinct affine points, i.e. the point $P$ is aligned to other $m$ distinct points of $\mathcal C$.

For a set of points $T\subseteq \vP^2(\vF_q)$, we say that a point $Q\in \vP^2(\vF_q)\setminus T$ is \emph{covered by  $T$} if there are $m$ points of $T$ aligned to $Q$.

We now show how to construct the set $S$ recursively so that $\mathcal C(\vF_q)\cup S$ is a complete $m$-arc of $\vP^2(\vF_q)$. First, set $A_0=\mathcal C(\vF_q)$. Now for $i\in \{1,\ldots,N\}$ set $A_i=A_{i-1}\cup H_i$ where $H_i$ satisfies the following
\begin{itemize}
\item[1)] $H_i\subset r_i$ and $|H_i|\leq m$;
\item[2)] $\nexists \; m+1$ points in $A_i=A_{i-1}\cup H_i$ that are aligned;
\item[3)] each point of $r_i$ is covered by $A_i$.
\end{itemize}
Such $H_i$ always exists: running over the points in $r_i\in \mathcal R$ we consider them one by one and add them to the arc until we get all points on $r_i$ covered. This procedure adds at most $m$ points, which constitute the set $H_i$. In other words, for every $i\in \{1,\ldots,N\}$, one must consider the tangent $r_i$, take a point and add it to the set if it is not covered by the set of points added until then. We observe now that $H_i$ cannot contain more than $m$ points because $r_i$ is a line (and therefore any $m$ points on $r_i$ cover all points in $r_i$). Once all points on the line $r_i$ are covered, one can move to $r_{i+1}$ continuing the same procedure.

Finally, consider $\ell_\infty$ the line at infinity in $\vP^2(\vF_q)$. With the same procedure, we can find at most $m$ points $P_1,\ldots,P_{m}\in\ell_\infty$ in such a way that
\begin{equation*}
A=A_N\cup\lbrace P_1,\ldots,P_{m} \rbrace
\end{equation*}
is also an $m$-arc. Set now
\begin{equation*}
S:=A\setminus \mathcal C(\vF_q).
\end{equation*}
Note that since we have added at most $m$ points for each tangent in $\mathcal R$ and for $\ell_\infty$, the size of $S$ is at most ${\mathcal O(m^5)}$. Obviously the set $A=C(\vF_q)\cup S$ is a complete $m$-arc of $\vP^2(\vF_q)$ by construction.
\end{proof}

\begin{remark}\label{rem:boundgenus}
Following \cite[Corollary 3.4]{Bartoli2022Complete}, since $p>m$, the explicit constant $c$ can be chosen to be
\begin{equation*}
c=9(g_K+g_L+m)^2(m!)^2=9m^2(m!)^2,
\end{equation*}
where $g_K$ and $g_L$ are the genera of $K$ and $L$ appearing in the proof of Theorem \ref{maintheorem}.
\end{remark}

Since a non-singular projective curve obviously verifies the hypohesis of Theorem \ref{maintheorem}, we have proven the following, which we state for conveniency of future reference.

\begin{theorem}
Let $p$ be a prime number and $q$ a power of $p$. Let $\mathcal C\subseteq \vP^2$ be a planar, irreducible, nonsingular projective curve of degree $m$ over a finite field $\vF_q$ such that $p>m$, and $\mathcal C$ is birationally equivalent to its dual $\mathcal C^*$ through the Gauss map. There exists an explicit constant $c$ independent of $q$, for which if $q>c$ then there exists a set of points $S\subseteq \vP^2(\vF_q)$ of size at most $O(m^5)$ such that $\mathcal C(\vF_q)\cup S$ is a complete $m$-arc of $\vP^2(\vF_q)$.
\end{theorem}

We now observe that our theory goes well beyond the restrictions given by Theorem \ref{maintheorem}. In fact, if one does not require the characteristic to be greater than the degree of the curve one has

\begin{theorem}\label{maintheoreminfl}
Let $p$ be a prime number and $q$ a power of $p$. Let $\mathcal C\subseteq \vP^2$ be a planar, irreducible, projective curve of degree $m$ over a finite field $\vF_q$, the singular points of $\mathcal C$ are contained in the line at infinity, and $\mathcal C$ is birationally equivalent to its dual $\mathcal C^*$ through the Gauss map. There exists an explicit constant $c$ independent of $q$, for which if $q>c$ then there exists a set of points $S\subseteq \vP^2(\vF_q)$ of size at most $mT+O(m^5)$ such that $\mathcal C(\vF_q)\cup S$ is a complete $m$-arc of $\vP^2(\vF_q)$, where $T$ is the number of inflection tangents to the curve $\mathcal C$.
\end{theorem}
\begin{proof}
The proof of this is completely identical with the wrinkle that one has to add separately $m$ points for every inflection tangent to the curve, where now the number of inflection tangents is not absolutely bounded. More precisely, this refined count is necessary because we are not anymore ensured that the number of inflection tangents is $O(m^2)$, as Lemma \ref{HessianTheorem} would prescribe in the case $p>m$.
Notice also that the genus of $M$ (on which the constant $c$ depends in the proof of Theorem  \ref{maintheorem}) can be bounded independently of $q$ simply using recursively Castelnuovo Inequality \cite[Theorem 3.11.3]{stichtenoth}.
\end{proof}
We will show an application of Theorem \ref{HessianTheorem} in the case of Artin-Schreier curves.

\section{Applications}\label{sec:applications}
To show the effectiveness of our construction, in this section we want to illustrate how a complete $m$-arc can be constructed using hyperelliptic curves, and Artin-Schreier curves of degree $m>p$.

\begin{proposition}
Let $p$ be an odd prime number, $q$ be a power of $p$, and $m\geq 3$ be a positive integer smaller than $p$.
Let $f(x)\in\vF_q[x]$ be a polynomial of degree $m$ with $m$ distinct roots. Let $\mathcal C$ be the projective closure of the affine curve defined by the equation $y^2-f(x)=0$. Then 
$\vF_q(\mathcal C)$ is isomorphic to $\vF_q(\mathcal C^*)$ via the Gauss map $\mathcal G$.
\end{proposition}
\begin{proof}

It is enough to show the equality in \eqref{eq:gaussmapfunc} using $g(x,y)=y^2-f(x)$. First, let us observe that

\begin{equation*}
 \vF_q(\mathcal C^*)\cong \vF_q\left(\frac{f'(x)}{2y},-\frac{xf'(x)}{2y}+y \right) \subseteq \vF_q(\mathcal C) = {\vF_q(x)[y]}/{\langle y^2-f(x) \rangle}.
\end{equation*}

From now on, we will omit the dependence on $x$ when it does not create confusion, and write $f$ and $f'$ instead of $f(x)$ and $f'(x)$. In the rest of the proof we show that $\vF_q(x)\subseteq \vF_q\left(\frac{f'}{2y},-\frac{xf'}{2y}+y \right)$ which directly implies that 
\[\vF_q\left(\frac{f'}{2y},-\frac{xf'}{2y}+y \right)=\vF_q(\mathcal C).\]
Notice that 
\[\left(\frac{f'}{2y}\right)^2=\frac{(f')^2}{4f}=:k,\]

\[ \frac{f'}{2y}\left(-\frac{xf'}{2y}+y\right)=-\frac{x(f')^2}{4f}+\frac{f'}{2}=\frac{-x(f')^2+2ff'}{4f}=:h\]
and that
\[\left(-\frac{xf'}{2y}+y \right)^2=\frac{(xf'-2f)^2}{4f}=:u\]
 
Let $a\in \vF_q$ be the leading coefficient of $f$. Observe that the leading coefficient of $-x(f')^2+2ff'$ is $-(a\deg(f))^2+2a^2\deg(f)=\deg(f)a^2(-\deg(f)+2)\neq 0$ because $f$ has degree larger than $2$ and the characteristic is larger than $\deg(f)$.
Now let us split two cases.

Case  (1): $x\nmid f$. In this case $\deg(h)=2\deg(f)-1$, because numerator and denominator are coprime, and $k$ has degree $2\deg(f)-2$.
By Luroth Theorem this shows that $\vF_q(h,k)=\vF_q(x)$, as we now explain. In fact, let $s=s_1/s_2\in \vF_q(x)$ (with $s_1,s_2\in \vF_q[x]$ and coprime) be such that $\vF_q(h,k)=\vF_q(s)$ and let $n=\deg(s)=\max\{\deg(s_1),\deg(s_2)\}=[\vF_q(x):\vF_q(s)]$. Then by the Tower Law we also have that every function in $\vF_q(s)$ has degree divisible by $n$, but since $\gcd(\deg(h),\deg(k))=1$, then $n=1$, and therefore $\vF_q(x)=\vF_q(s)$, concluding the proof for this case.

Case (2): $x\mid f$. Rewrite $u$ as 
\[u=\frac{x(f'-2(f/x))^2}{4(f/x)}\]
 and therefore observe that $\deg(u)=2\deg(f)-1$ because the leading coefficient of $f'-2(f/x)$ is $a\deg(f)-2a\neq 0$, $f$ is squarefree, and the numerator and denominator are coprime because $f$ and $f'$ are coprime. But $\deg(k)=2\deg(f)-2$ which concludes the proof using Luroth's theorem with the same argument as Case (1).

\end{proof}

The following is immediate to see using Theorem \ref{maintheorem}.
\begin{corollary}
Let $m\geq 3$ be a positive integer and $p$ an odd prime greater than $m$.
There exists an absolute constant $c$, depending only on $m$, such that for any $q=p^\ell>c$ power of $p$, any hyperelliptic curve $\mathcal H:y^2-f(x)=0$ over $\vF_q$ of degree $m$ gives rise to a complete $m$-arc of size $|\mathcal H(\vF_q)|+\mathcal O(m^5)$ in the sense of Theorem \ref{maintheorem}. 
\end{corollary}
\begin{proof}
It is enough to check that the hypotheses of Theorem \ref{maintheorem} are verified.
\end{proof}

The next example allows us to construct a complete $m$-arc on $\vF_{q^2}$ with $q^2-2gq+c$ points, where $c$ is a constant that does not depends on $q$, and $g$ is the genus of the hyperelliptic curve:
\begin{equation*}
g=\begin{cases}
(m-1)/2 \;\;\;\;\;\;\;\; \text{if} \; m\equiv 1\mod 2, \\
(m-2)/2 \;\;\;\;\;\;\;\; \text{if} \; m\equiv 0\mod 2.
\end{cases}
\end{equation*}

\begin{example}\label{example}
We illustrate now our method by producing a complete $m$-arc of small size for a very well known family of hyperelliptic curves.
Let $q$ be an odd prime power and $m$ a positive integer that divides $q+1$. Then, by \cite[Theorem 1]{TAFAZOLIAN20121528}, the hyperelliptic curve $\mathcal{C}$ given by $y^2=x^m+1$ is maximal over $\vF_{q^2}$, i.e. the number of places of degree $1$ (i.e. defined over $\vF_{q^2}$) of $\vF_{q^2}(\mathcal C)=\vF_{q^2}(x,y)$ (with $y^2=x^m+1$) is $\#\mathcal C(\vF_{q^2})=q^2+1+2gq$ where $g$ indicates the genus of $\mathcal C$. 
Notice that the affine places of degree $1$ of $\vF_{q^2}(\mathcal C)$ are in natural bijective correspondence with the $\vF_{q^2}$-rational points of the affine points of $\mathcal C$, as the only singular point is at infinity.

We note that, as $x$ varies in $\vF_{q^2}$, $x^m+1$ essentially covers the maximal number of squares. Conversely, for every non-square $\xi\in\vF_{q^2}$, the curve $\mathcal C'$ defined by $y^2=\xi(x^m+1)$ covers fewer squares and can be used to construct a $m$-arc with less points. We now give an estimate of the number of points of $\mathcal C'$.

First of all, notice that if a point $(x,0)$ belongs to $\mathcal C$, it also belongs to $\mathcal C'$. Moreover, since $m|q+1$ and $q-1$ is even, $2m|q^2-1$. Thus the polynomial $x^{2m}-1$ splits completely over $\vF_{q^2}$ and so $x^m+1$ does. This means that the number of points of the form $(x,0)$ is exactly $m$.

The $\vF_{q^2}$ affine rational points $(x_0,y_0)$ of $\mathcal C$ such that $y_0\neq 0$ are at least 
\[q^2+1+2gq-\#\mathcal P_\infty-m\]
where $\#\mathcal P_\infty$ refers to the number of places at the infinity of $\vF_q(\mathcal C)$. Therefore, the $x_0$'s such that $x_0^m+1$ is a square different from zero are at least 
\[\frac{q^2+1+2gq-\#\mathcal P_\infty-m}{2}\]

We can now count the $\vF_{q^2}$ affine rational points of $\mathcal C'$, which are
\begin{equation*}
m+2\Big( q^2-\frac{q^2+1+2gq-\#\mathcal P_\infty -m}{2} \Big) = q^2 - (1+2gq-\#\mathcal P_\infty) +2m.
\end{equation*}

Therefore, the complete $m$-arc obtained from $\mathcal C'$ is made of $q^2-2gq+\mathcal O(m^5)$ points. Of course, if one works out exactly the constants and the number of 
bitangents it follows that the implied constants and the $O(m^5)$ are loose bounds and one can do better. Nevertheless, this shows that the size  of the complete $m$-arc $\mathcal A_{q^2}$ constructed from these hyperelliptic curves satisfies $\# \mathcal A_{q^2}-q^2\rightarrow -\infty$ for $q\rightarrow \infty$. As a byproduct, this also removes the restriction on the existence of the prime $r$ in \cite[Proposition 6.4]{Bartoli2022Complete}.
\end{example}

\begin{remark}
So far in the literature, it has been shown how complete $m$-arcs can be constructed for $m=2,3,4$ and $m\geq 8$ (see \cite{Bartoli2022Complete}) using certain curves. This example, and more generally the theory developed in this work, can be used to construct $m$-arcs for arbitrary values of $m\geq 3$ and arbitrary curve verifying generic conditions. As a byproduct, we have covered the cases when $m=5,6,7$ that were missing in the literature. 

\end{remark}

Using Theorem \ref{maintheoreminfl}, we now show that Artin-Schreier curves of degree $m>p$ can be used to construct an $m$-arc.

\begin{theorem}
Let $\vF_q$ be a finite field of odd characteristic $p$ and $m$ be a positive integer larger than $p$.
Let $f\in \vF_q[x]$ be such that $f\neq z^p-z$ for any $z\in \vF_q[x]$, and $f$ squarefree of degree $m$ such that $f'\not\in \vF_q[x^p]$.
Let $\mathcal C$ be the projective closure of the affine curve defined by $g(x,y)\coloneqq y^p-y-f(x)=0$. Then there exists an absolute constant $c$, depending only on $m$, such that if $q>c$, then there exists a complete $m$-arc with $|\mathcal C(\vF_q)|+O(m^5)$ points, where the implied constant is independent of $q$.
\end{theorem}
\begin{proof}
First, we will prove that 
\[
\vF_q(x,y)=\vF_q\left(\frac{\partial_x g(x,y)}{\partial_y g(x,y)}, x\frac{\partial_x g(x,y)}{\partial_y g(x,y)}+y \right).
\]
verifying the hypothesis of Lemma \ref{lemma0}. Therefore, we need to show that 

\[ \vF_q(f',xf'+y)=\vF_q(x,y),\]
where $g(x,y)=0$.
By computing
\[(xf'+y)^p-(xf'+y)=x^p(f')^p-xf'+f=:h.\]
We realize that if we prove that $\vF_q(f',h)=\vF_q(x)$, then $\vF_q(x)\subseteq\vF_q(f',xf'+y)$ which implies $\vF_q(f',xf'+y)\supseteq \vF_q(x,y)$ and therefore the claim.

Since $\vF_q(f',h)\subseteq \vF_q(x)$, by Luroth's Theorem there esists 
$s\in \vF_q(x)$ such that 
$\vF_q(f',h)=\vF_q(s)$. 
In particular, this implies that, for some $u,v\in\vF_q(x)$, we have that $u(s)=f'$ and $v(s)=h$, which forces 
\[x^pu(s)^p-xu(s)+f=v(s).\]
By deriving both sides we get
\[-u(s)-xu'(s)s'+u(s)=v'(s)s'\]
and therefore
\[(xu'(s)+v'(s))s'=0\]
so either $x=-v'(s)/u'(s)$, in which case we are done, or $s'=0$, but then $s=\ell(x^p)\in \vF_q(x^p)$ and therefore $f'=u(\ell(x^p))$, contradicting the hypothesis $f'\not\in \vF_q[x^p]$.

We now count the number of inflection tangents. Notice that for every inflection point there exists exactly one inflection tangent. Therefore, one can bound the number of inflection tangents by the number of inflection points. 

We are now going to characterize the inflection points.
Suppose $P_0=(x_0,y_0)$ is an affine inflection point of $\mathcal C$ and let $\ell_{P_0}$ be the affine tangent line to $\mathcal C$ at such a point, namely \[ \ell_{P_0}: \; -f'(x_0)(x-x_0)+y_0-y=0.\]
Now, intersecting $g(x,y)=0$ with $\ell_{P_0}$, one gets \[ -f'(x_0)^p(x-x_0)^p+y_0^p + f'(x_0)(x-x_0)-y_0 -f(x)=0 \] and since $y_0^p-y_0=f(x_0)$ we have \[ -f'(x_0)^p(x-x_0)^p+ f'(x_0)(x-x_0)+f(x_0) -f(x)=0. \]
Writing $t\coloneqq x-x_0$ one can rewrite the previous equation as follows
\begin{equation}\label{eq-tSobsitution}
-f'(x_0)^p t^p+ f'(x_0)t + f(x_0) -f(t+x_0)=0.
\end{equation}
Now notice that, for suitable $a_0,\ldots,a_m\in\vF_q$, one always have
\begin{equation}\label{eq-fPolynomial}
	\begin{split}
		f(t+x_0)= & \sum_{i=0}^m a_i(t+x_0)^i= \sum_{i=0}^m \sum_{j=0}^i a_i \binom{i}{j} x_0^j t^{i-j} = \\
		= & f(x_0) + f'(x_0)t + \frac{1}{2}f''(x_0)t^2 + \sum_{i=0}^m \sum_{j=0}^{i-3} a_i \binom{i}{j} x_0^j t^{i-j}
	\end{split}
\end{equation}
Using (\ref{eq-fPolynomial}) and (\ref{eq-tSobsitution}) one gets
\begin{equation}\label{eq-final}
f'(x_0)^p t^p+ \frac{1}{2}f''(x_0)t^2 + \sum_{i=0}^m \sum_{j=0}^{i-3} a_i \binom{i}{j} x_0^j t^{i-j} = 0.
\end{equation}
However, since $P_0=(x_0,y_0)$ is an inflection point, the order of tangency must be at least $3$, namely in the equation (\ref{eq-final}) the term $f''(x_0)$ should be zero. This means that the inflection points are contained in the set of solutions of $g(x_0,y_0)=0$ and $f''(x_0)=0$.
Since $f'\not\in \vF_q[x^p]$, then $f''$ is not identically zero, which implies that there are at most $\deg(f'')p=(m-2)p$ inflection tangents. Since $p<m$ we have that the number of inflection tangents is at most $m^2$.
We now apply directly Theorem \ref{maintheoreminfl}.
\end{proof}

\section{Acknowledgements}

This work was supported by the National Science Foundation under Grant No 2127742.
We would like to thank Gabriele Mondello for providing deep geometric insight that helped us to simplify some proofs in the paper.

\nocite{*}
\bibliographystyle{siam}
\bibliography{bibliography}
\end{document}